\documentclass[11pt,leqno]{article}
\usepackage{amsthm,amsfonts,amssymb,amsmath,oldgerm}
\usepackage{epsfig}
\numberwithin{equation}{section}
\usepackage[thinlines]{easybmat}

\newcommand{\btheta}{{\bar \theta}}

\def\eps{\varepsilon }

\newcommand\R{\mathbb R}

\def\eps{\varepsilon}

\newcommand\br{\begin{remark}}
\newcommand\er{\end{remark}}
\newcommand\bp{\begin{pmatrix}}
\newcommand\ep{\end{pmatrix}}
\newcommand\be{\begin{equation}}
\newcommand\ee{\end{equation}}
\newcommand\ba{\begin{equation}\begin{aligned}}
\newcommand\ea{\end{aligned}\end{equation}}

\newcommand{\bap}{\begin{app}}
\newcommand{\eap}{\end{app}}
\newcommand{\begs}{\begin{exams}}
\newcommand{\eegs}{\end{exams}}
\newcommand{\beg}{\begin{example}}
\newcommand{\eeg}{\end{exaplem}}
\newcommand{\bpr}{\begin{proposition}}
\newcommand{\epr}{\end{proposition}}
\newcommand{\bt}{\begin{theorem}}
\newcommand{\et}{\end{theorem}}
\newcommand{\bc}{\begin{corollary}}
\newcommand{\ec}{\end{corollary}}
\newcommand{\bl}{\begin{lemma}}
\newcommand{\el}{\end{lemma}}
\newcommand{\bd}{\begin{definition}}
\newcommand{\ed}{\end{definition}}
\newcommand{\brs}{\begin{remarks}}
\newcommand{\ers}{\end{remarks}}

\newtheorem{theo}{Theorem}[section]

\newtheorem{cor}[theo]{Corollary}

\newtheorem{exams}[theo]{Examples}

\numberwithin{equation}{section}

\newcommand{\RR}{{\mathbb R}}

\newcommand{\CC}{{\mathbb C}}

\newcommand{\Id}{{\rm Id }}
\newcommand{\Range}{{\rm Range }}

\newcommand{\sgn}{\text{\rm sgn}}
\newtheorem{theorem}{Theorem}[section]
\newtheorem{proposition}[theorem]{Proposition}
\newtheorem{corollary}[theorem]{Corollary}
\newtheorem{lemma}[theorem]{Lemma}
\newtheorem{definition}[theorem]{Definition}

\newtheorem{example}[theorem]{Example}
\newtheorem{remark}[theorem]{Remark}

\newcommand\cF{{\cal  F}}

\newcommand{\tr}{\,\mbox{\rm tr}}

\newcommand{\bi}{\bibitem}

\newcommand{\beq}{\begin{equation}}
\newcommand{\eeq}{\end{equation}}

\title{
Convergence as period goes to infinity of spectra of periodic traveling waves toward essential spectra of a homoclinic limit
}

\author{\sc \small
Zhao Yang\thanks{
Indiana University, Bloomington, IN 47405;
yangzha@indiana.edu: Research of Z.Y. was partially supported
under NSF grant no. DMS-0300487 and an Indiana University Research Assistantship.
}
and
Kevin Zumbrun\thanks{Indiana University, Bloomington, IN 47405;
kzumbrun@indiana.edu: Research of K.Z. was partially supported
under NSF grant no. DMS-0300487.
 }}
\begin{document}

\maketitle


\begin{center}
\end{center}


\begin{abstract}
We revisit the analysis by R.A. Gardner of convergence of spectra of periodic
traveling waves in the homoclinic, or infinite-period limit, extending
his results to the case of essential rather than point spectra of the limiting homoclinic wave.
Notably, convergence to essential spectra is seen to be of algebraic rate with respect 
to period as compared to the exponential rate of convergence to point spectra.
In the course of the analysis, we show not only convergence of spectrum
but also convergence of an appropriate renormalization of the associated
periodic Evans function to the Evans function for the limiting homoclinic wave,
a fact that is useful for numerical investigations.
\end{abstract}

\section{Introduction}

In this note, using asymptotic Evans function techniques
like those introduced for the study of homoclinic and heteroclinic traveling waves in \cite{PZ,Z1,Z2},
we build on the pioneering analysis of R.A. Gardner \cite{G1,G2} of convergence of spectra of periodic traveling
waves in the infinite-period, or ``homoclinic'', limit,
extending his results to the case that the limiting homoclinic spectra are of essential rather than point spectrum type.

Under quite general conditions, Gardner showed that loops of essential periodic spectra 
bifurcate from isolated point spectra $\lambda_0$ of the limiting homoclinic wave.
Indeed, it is readily seen that, on compact sets bounded away from regions of essential homoclinic spectrum, 
periodic spectra converge as period $X\to \infty$ at exponential rate $O(e^{-\eta X})$, $\eta>0$ 
to the point spectra of the limiting homoclinic; see \cite{SS,OZ1}, or Section \ref{s:convergence} below.

In the standard case arising generically for reaction diffusion systems
of a limiting homoclinic wave with strictly stable essential
spectrum, or ``spectral gap'', and a single isolated eigenvalue at $\lambda=0$ associated with
translational invariance of the underlying equations, this reduces the study of periodic stability in
the large-period limit to asymptotic analysis of the loop of ``critical'' periodic spectra bifurcating from 
the neutral eigenvalue $\lambda=0$.
For, recall that linearized and nonlinear stability have been shown in quite general circumstances 
to follow from the ``dissipative spectral stability'' condition of Schneider: that periodic spectra 
move into the stable half plane at quadratic rate in the associated Bloch-Floquet number as the Bloch number is
varied about zero \cite{S1,S2,JZ1,JZ2,JZN,JNRZ1}.
This problem was resolved definitively by Sandstede and Scheel in \cite{SS}, essentially closing the
question of large-period periodic stability in the case of a spectral gap.

However, there are interesting cases arising in systems with conservation laws,
notably for models of elasticity and thin film flow \cite{OZ1,OZ2,JZN,BJNRZ3}
of families of periodic waves for which the spectral gap condition is not satisfied
in the homoclinic limit, $\lambda=0$ being an eigenvalue embedded in the essential spectrum.
In particular, for thin film flows, the homoclinic limit typically has unstable
essential spectrum branching from the origin, and it is spectra bifurcating from
this essential spectrum rather than the embedded eigenvalue at $\lambda=0$ that appears to
dominate the stability behavior of nearby periodic waves; see the discussion of \cite{BJNRZ4,BJRZ}.
This motivates our study here of convergence in the vicinity of essential spectra, both to 
essential spectra themselves and to eigenvalues embedded in essential spectrum,
to neither of which cases Gardner's original analysis applies.

Recall \cite{He,GZ} that the essential spectrum of a homoclinic traveling wave is given by the
union of algebraic curves $\lambda=\lambda_j(k)$ obtained from the dispersion relation of the
(constant-coefficient) linearization of the governing evolution equation about the endstate 
$u_\infty=\lim_{x\to \pm \infty} \bar u(x)$ of the 
homoclinic profile $\bar u(\cdot)$, where $k\in \R$ denotes Fourier frequency,
corresponding to the (entirely essential) spectra of the constant solution $u(x,t)\equiv u_\infty$.
Thus, the generic situation in the context of essential spectrum, analogous to an isolated eigenvalue in
the point spectrum context considered by Gardner, is a point $\lambda$ lying
on a single curve $\lambda_j$, corresponding to a single nondegenerate root $k_*$ of $\lambda_j(k)=0$.
This is closer in nature to the (also entirely essential) spectra of periodic waves than is the case of
an isolated eigenvalue; indeed, it is the spectra of the constant periodic solution $u(x,t)\equiv u_0$,
with Fourier frequency $k$ corresponding to $\gamma$-value $\gamma=e^{ikX^\eps}$ in the 
notation of Section \ref{s:prelims}.

Similarly as in \cite{G2}, our analysis is carried out by examination
of the associated periodic Evans functions $E^\eps(\lambda, \gamma)$
introduced by Gardner \cite{G1}, $\lambda, \gamma \in \CC$, $|\gamma|=1$,
an analytic function whose zeroes $\lambda$ coincide with the spectrum
of the linearized operator about the wave, where $\eps\to \RR$ indexes
the family of periodic waves converging as $\eps\to 0$ to a homoclinic,
or solitary wave, profile.
However, differently from the approach of \cite{G1,G2}, our results are obtained not by topological considerations,
but, similarly in \cite{Z1,Z2,JZ3}, by demonstration of convergence, at exponential rate
$O(e^{-\eta X^\eps})$, of a suitably rescaled version of
the sequence of periodic Evans functions $E^\eps(\lambda, \gamma)$ and a sequence of functions
interpolating between different versions of the homoclinic Evans function $D^0(\lambda)$ 
defined on various components of the complement of the union of curves $\lambda_j(\cdot)$ composing the 
homoclinic essential spectrum, with transition zones of scale $\sim 1/X^\eps\to 0$ around isolated points
$\lambda=\lambda_j(k)$.

Away from the homoclinic essential spectrum, this 
reduces to the simpler computation an appropriate renormalization $D^\eps(\lambda, \gamma)$ of $E^\eps$
converges to the homoclinic Evans function $D^0(\lambda)$ at exponential rate, recovering and further 
illuminating the original result of Gardner \cite{G2} that the zero-set of $E^\eps(\cdot, \xi)$ converges for
each $\xi$ to the zero-set of $D^0$, at exponential rate; see Section \ref{s:convergence}.
This convergence is potentially useful in numerical investigations,
as the basis of numerical convergence studies in this singular, hence numerically sensitive, limit.
See, e.g., the applications in \cite{BJNRZ2}, as discussed in \cite[Appendix D, pp. 70--72]{BJNRZ2}. 

Near isolated arcs of curves $\lambda_j$ of the homoclinic essential spectrum, as described above,
for which the bordering homoclinic Evans functions do not vanish (in particular precluding embedded eigenvalues),
we find that the zero-set of $E^\eps$, comprising curves of periodic essential spectrum, converges {\it not to
a single point but to a full arc of $\lambda_j$}, 
and at {\it algebraic rather than exponential rate}; see Section \ref{s:essential}.
In the case of an isolated arc with a single embedded eigenvalue, we find as might be guessed
that the periodic spectra comprise two curves: a loop converging exponentially to the isolated eigenvalue, 
and a curve converging algebraically to the arc $\lambda_j$; see Section \ref{s:embedded}.
The method of analysis is general, and should extend to other, more degenerate cases, at
the expense of further effort/computation.

Our results apply in particular to the Saint Venant equations of inclined shallow water flow studied in \cite{BJNRZ4,BJRZ},
verifying instability of periodic waves in the homoclinic limit by consideration of spectra bifurcating from 
unstable essential spectrum of the limiting homoclinic.  A very interesting open problem 
is to carry out an analysis like that of \cite{SS} determining separately the stability of spectra
bifurcating from the embedded eigenvalue at $\lambda=0$.
An insteresting related problem is to verify the heuristic picture of ``metastable'' behavior conjectured in \cite{BJRZ}, 
deducing stability for large but not infinite-period waves 
based on properties of an essentially unstable homoclinic limit with stable point spectrum.

\section{Preliminaries}\label{s:prelims}

Following Gardner \cite{G2}, we consider a family of periodic traveling-wave solutions 
\be\label{e:tw}
u(x,t)=\bar u^\eps(x-c^\eps t),
\quad
\bar u^\eps(x+X^\eps)=\bar u^\eps(x)
\ee
of a family of PDEs $u_t=\cF^\eps (\partial_x, u)$
with smooth coefficients,
converging as $\eps \to 0$ to a solitary-wave solution
$\bar u^0$, or homoclinic orbit of the associated traveling-wave ODE
$-c^0 \partial_x u=\cF^0(u)$,
as meanwhile $X^\eps\to \infty$.
Taking without loss of generality $c^\eps\equiv 0$ (changing
to co-moving coordinates $\tilde x=x-c^\eps t$), we investigate
stability of the equilibria $\bar u^\eps$, $\cF^\eps(\bar u^\eps)=0$, 
through the study of the spectra $\lambda$ of the associated family of eigenvalue ODEs 
\be \label{e:eig}
\lambda u=L^\eps u:=d\cF(\bar u^\eps)u,
\ee
with an eye toward relating the spectral properties of periodic waves $\bar u^\eps$
as $\eps\to 0$ to those of the limiting homoclinic $\bar u^0$.

Assume as in \cite{G1,G2} that \eqref{e:eig} may be written as
a first-order system
\be\label{e:firstorder}
W'=A^\eps(x,\lambda)W
\ee
in an appropriate phase variable $W$, where $A^\eps$ is analytic in 
$\lambda$, $C^1$ in $x$, and continuous in $\eps$ for $\eps>0$.
Then, the spectrum of the periodic waves $\bar u^\eps$, $\eps>0$ is
made up of essential spectra given \cite{G1} by the union of
$\gamma$-eigenvalues $\lambda$ consisting of zeroes of the
{\it periodic Evans function}
\be\label{pevans}
E^\eps (\lambda, \gamma):=\det(\Psi^\eps(X^\eps,\lambda)-\gamma \Id),
\ee
where $\Psi^\eps(x,\lambda)$ denotes the solution operator of \eqref{e:firstorder},
with $\Psi^\eps(0,\lambda)=\Id$ and $\gamma\in \CC$ with $|\gamma|=1$.
A $\gamma$-eigenvalue of particular importance is the $1$-eigenvalue
$\lambda=0$ associated with eigenfunction $\partial_x \bar u$
corresponding to instantaneous translation, arising through 
translation-invariance of the underlying PDE.
In the case that $\cF^\eps$ is divergence-form, there exist
other important $1$-eigenvalues corresponding to variations
along the manifold of nearby $X^\eps$-periodic solutions,
which in this case has dimension $\dim u + 1>1$ \cite{OZ1,JZ1}.

As shown in a variety of settings (see \cite{S1,S2,JZ1,JZ2,JZN,BJNRZ1,BJNRZ2,JNRZ1}
and references therein), linearized and nonlinear modulational stability
are implied by the properties:

(D1) the multiplicity of the $1$-eigenvalue $\lambda=0$ is equal to
the dimension $d$ of the manifold of nearby $X^\eps$-periodic solutions
(in the typical case considered by Gardner \cite{G2}, $d=1$).

(D2) other than the $1$ eigenvalue $\lambda=0$, there are no other
$\gamma$-eigenvalues with $\Re \lambda \ge 0$.

(D3) parametrizing $\gamma=e^{ikX}$, $\Re \lambda \le -\eta k^2$
for $0\le kX\le 2\pi$, for some $\eta>0$. 

\noindent 
Accordingly, these are the spectral properties that we wish to investigate.
In particular, note that (D3) concerns not only location, but 
curvature of the spectral loop through $\lambda=0$.

Conditions (D1)--(D2) are easily seen to be necessary for 
linearized modulational stability, while condition (D3) implies
a Gaussian rate of time-algebraic decay sufficient to close
a nonlinear iteration; see \cite{S1,S2,OZ2,JZ1,JZ2} for further discussion.

\subsection{Assumptions}\label{s:assumptions}

Loosely following \cite{G2}, we assume, for $|\lambda|\leq M$:
\medskip

(H1) $X^\eps\to \infty$ as $\eps \to 0$.

(H2) $|A^0(x,\lambda)-A^0_\infty(\lambda)|\le C(M)e^{-\nu |x|}$,
for some $C(M),\nu>0$.

(A3) 
$|\bar u^\eps(x)-\bar u^0(x)|\le \delta(\eps)$
for $|x|\le \frac{X^\eps}{2}$, with $\delta(\eps)\to 0$ as $\eps\to 0$.

\medskip
\noindent
In order to obtain the quantitative estimates we require, we augment (A3) with
\be\label{delbd}
\delta(\eps)\le C e^{-\bar\theta X^\eps/2}
\; \hbox{\rm  for some } \;  \bar\theta>0.
\ee

\br\label{h3rmk1}
Condition \eqref{delbd} is an additional assumption beyond those made
in \cite{G2}.
However, it follows from (A3) in the standard case that the vertex
$\bar u_\infty^0=\bar u(\pm\infty)$ of the limiting homoclinic 
is a hyperbolic rest point of the traveling-wave ODE, under the generically satisfied transversality condition that the 
associated Melnikov separation function be full rank with respect to $\eps$, as do (H1)-(H2) as well,
with $\nu=\bar \theta=\alpha$, were $\alpha$ is the minimum growth/decay rate of the linearized equations about
$u_\infty^0$; see
\cite[Prop 5.1 pp. 166-167]{SS}. 
Thus, Gardner's original condition (A3) is the main assumption in practical terms.
\er

\br\label{h3rmk2}
In the planar Hamiltonian traveling-wave ODE setting,
for which all periodics and the limiting homoclinic lie in the same
phase portrait of a single traveling-wave ODE,
setting $\eps$ to be the distance of $\bar u^\eps(\cdot)$ from the saddle-point
$\bar u^0_\infty$,
one may compute more or less explicitly 
that $X^\eps \sim c\log \eps^{-1}$ and $\delta(\eps)\le C\eps$,
with $|\bar u^\eps-\bar u^0|\le C\eps^2$ away from $\bar u^0_\infty$.
This gives another class of interesting examples to which our assumptions apply.
\er

As a consequence of (A3) we obtain for $|\lambda|\le M$, 
$ |A^\eps(x,\lambda)-A^0(x,\lambda)|\le C(M)\delta(\eps) $
as in assumption (iii) of \cite{G2}, p. 152,
yielding together with \eqref{delbd}:

\medskip

(H3)
$|A^\eps(x,\lambda)-A^0(x,\lambda)|\le C(M)
e^{-\bar\theta X^\eps/2}$
for $|\lambda|\le M$, $|x|\le \frac{X^\eps}{2}$, 
and $\bar\theta >0$.

\medskip
\noindent 
Hereafter, we drop the motivating assumption (A3) and work 
similarly as in \cite{G2} with 
hypotheses (H1)--(H3) on the first-order eigenvalue system
\eqref{e:firstorder} alone.

\section{The homoclinic and rescaled periodic Evans functions}\label{s:evandefs}

We begin by formulating the homoclinic and periodic Evans functions following the approach of
\cite{MZ,Z1,Z2}, in a way that is particularly convenient for their comparison.

\subsection[Reduction to constant coefficients]
{Reduction to constant coefficients}\label{s:reduction}

Adapting the asymptotic ODE techniques developed in 
\cite{MZ,PZ,Z1,Z2} for problems on the half-line
(see Appendix \ref{asymptotic}), we obtain the following quantitative
description relating \eqref{e:firstorder} 
to a constant-coefficient version of the homoclinic eigenvalue problem $W'=A^0(x,\lambda)W$.

\bl\label{conjlem}
Assuming (H1)--(H3), for each $\eps\ge 0$, 
there exist in a neighborhood of any $|\lambda_0|\le M$
bounded and uniformly invertible linear transformations $P^\eps_+(x,\lambda)$
and $P_-^\eps(x,\lambda)$ defined on $x\ge 0$ and $x\le 0$, respectively,
analytic in $\lambda$ as functions into $L^\infty [0,\pm\infty)$, such that,
for any $0<\bar \eta<min(\bar{\theta},\nu)$, $\bar \theta$, $\nu$  as in (H2)-(H3),
some $C>0$, and $|x|\le \frac{X^\eps}{2}$,
\begin{equation}
\label{Pdecay} 
P_\pm^\eps(\pm X^\eps/2)=\Id,
\end{equation}
\begin{equation}
\label{cPdecay} 
| (P^\eps-P^0)_\pm |\le C e^{-\bar\eta X^\eps/2}
\quad
\text{\rm for }  \; x\gtrless 0,
\end{equation}
and the change of coordinates $W=:P^\eps_\pm Z$ reduces \eqref{e:firstorder} 
to the constant-coefficient system
\begin{equation}
\label{glimit}
Z'=A^0_\infty Z,
\quad 
\text{\rm for } \; x\gtrless 0
\; \hbox{\rm and } \;
|x|\le \frac{X^\eps}{2}.
\end{equation}
\el

\begin{proof}
Extending $A^\eps(x,\lambda)$ by value $A^0_\infty$
for $|x| > \frac{X^\eps}{2}$, we obtain a modified family of 
coefficient matrices agreeing with $A^\eps$ on $|x|\le \frac{X^\eps}{2}$
and satisfying 
$$
	|(A^\eps(x,\lambda)-A^\eps_\infty ) -(A^0(x,\lambda)-A^0_\infty)|=
	|A^\eps(x,\lambda)-A^0(x,\lambda)| \le C(M) e^{-\bar \theta X^\eps/2}
$$
for $|x|\leq X^\eps/2$, and 
$$
	|(A^\eps(x,\lambda)-A^\eps_\infty ) -(A^0(x,\lambda)-A^0_\infty)|=
	|A^0(x,\lambda)-A^0_\infty | \le C(M) e^{- \nu |x|} 
$$
for $|x|\geq X^\eps/2$, yielding for all $x$ the estimate
\ba\label{a2}
	|(A^\eps(x,\lambda)-A^\eps_\infty ) -(A^0(x,\lambda)-A^0_\infty)| \le C(M)\delta_2(\eps) e^{- \sigma |x|} 
		\ea
for $\delta_2(\eps):=e^{-\bar \eta X^\eps/2}$, $0<\sigma < \min\{\nu, \bar \theta \}-\bar \eta$, and, trivially,
\be\label{a2'}
|A^\eps_\infty  -A^0_\infty|=0\leq  C(M)\delta_2(\eps).
\ee
Likewise, we have
\ba\label{a1}
|A^\eps(x,\lambda)-A^\eps_\infty(\lambda)|&= |A^\eps(x,\lambda)-A^0_\infty(\lambda)|\\
					  & \leq
|A^\eps(x,\lambda)- A^0(x,\lambda)|+ |A^0(x,\lambda)- A^0_\infty(\lambda)|\\
&\leq
2C(M) e^{- \min\{\nu, \bar \theta \} |x|},
\ea
hence also, by $\sigma< \min\{\nu, \bar \theta \}$, evidently
\be\label{weaker}
|A^\eps(x,\lambda)-A^\eps_\infty(\lambda)|= |A^\eps(x,\lambda)-A^0_\infty(\lambda)|\leq 2C(M) e^{- \sigma |x|}.
\ee

Using \eqref{a1} and applying Lemma \ref{conjlemapp}, Appendix \ref{asymptotic}, with $p=\eps$
and $\theta = \min\{\nu, \bar \theta \}$, we obtain 
$| P_\pm^\eps-\Id |\le C e^{- \bar\eta |x|}$ for $x\gtrless 0$.
Moreover, by Remark \ref{pfrmk},
$(P^\eps_\pm) '= A^p P^\eps_\pm  - P^\eps_\pm  A^\eps_\pm$ and $P^\eps_\infty=\Id$,
yielding $(P^\eps_\pm)'=0$ for $|x|\geq X^\eps/2$, and therefore \eqref{Pdecay}.
Finally, using \eqref{a2}-\eqref{a2'}, \eqref{weaker}, we obtain \eqref{cPdecay}
by Lemma \ref{evanslimit}, Appendix \ref{asymptotic}, 
with $p=\eps$, $\delta(p)=\delta_2(\eps)$, and $\theta= \sigma$.
\end{proof}

\subsection{The homoclinic Evans function}\label{s:hom}

Away from a finite set of curves $\lambda_j(k)$ determined
by the dispersion relation $ik\in \sigma(A^0_\infty(\lambda))$,
$k\in \RR$, where $ \sigma$ denotes spectrum, the eigenvalues
of $A^0_\infty$ have nonvanishing real part.
Denote by $\Lambda_r$ the open components of 
$\CC\setminus \{\lambda_j(k)\}$.
We refer to $\Lambda_r$ as the {\it domains of hyperbolicity} 
of $A^0_\infty$.
Denote by $n_r$ the number of negative real part eigenvalues of
$A^0_\infty$.

\begin{definition}[\cite{GZ,BJRZ}]\label{d:hom}
On each domain of hyperbolicity $\Lambda_r$, the homoclinic Evans
function is defined as
\ba\label{e:hevans}
D^0_r(\lambda)&:= 
\frac{ \det (R^-,R^+)|_{x=0}}
{\det (R^-_\infty,R^+_\infty)}
=
\frac{\det (P^0_-R^-_\infty,P^0_+R^+_\infty)}
{\det (R^-_\infty,R^+_\infty)},
\ea
where $R^-_\infty$ is any matrix whose columns are a basis for the unstable subspace
of $A^0_\infty$,
$R^+_\infty$ is any matrix whose columns are a basis for the stable subspace
of $A^0_\infty$, and $R^-(x):=P^0_-(x)e^{A^0_\infty x}R^-_\infty$ and $R^+(x):=
P^0_+(x)e^{A^0_\infty x}R^+_\infty$ are matrices whose columns are bases for the subspaces
of solutions of \eqref{e:firstorder} decaying as $x\to -\infty$
and $x\to +\infty$, respectively.
\end{definition}

Evidently, each $D^0_r$ is analytic in $\lambda$ on $\Lambda_r$, and
vanishes at $\lambda_0\in \Lambda_r$ if and only if $\lambda_0$ is an eigenvalue of
$L^0$. Moreover, it can be shown with in great generality 
that its zeros correspond
in multiplicity with the eigenvalues of $L^0$; see \cite{GZ,MaZ}, 
and references therein.

\subsection{The rescaled periodic Evans function}\label{s:bal}
Following \cite{G2}, we note that, by Abel's formula,
$E^\eps(\lambda, \gamma)$ may be written alternatively as 
$$
E^\eps(\lambda, \gamma)=
\tilde E^\eps(\lambda, \gamma) e^{\int_0^{X^\eps/2} 
\tr A^\eps(\lambda, y) dy},
$$
where
\ba
\tilde E^\eps(\lambda, \gamma)&:=\det (\Psi^\eps(0,\lambda)\Psi^\eps(-X^\eps/2,\lambda)^{-1}
- \gamma \Psi^\eps(0,\lambda)\Psi^\eps(X^\eps/2,\lambda)^{-1})
\ea
is a ``balanced'' periodic Evans function defined symmetrically
about $x=0$ similarly as the homoclinic Evans function.

\begin{definition}\label{e:balD}
On $\Lambda_r$, we define the {\it rescaled balanced
periodic Evans function} as
\be\label{e:rescaled}
D^\eps_r(\lambda, \gamma):=
	e^{-\tr A^0_\infty \Pi_u X^\eps/2} e^{\tr A^0_\infty \Pi_s X^\eps/2} (-\gamma)^{-n_r}
\tilde E^\eps(\lambda, \gamma) ,
\ee
where $\Pi_u$ and $\Pi_s$ denote the unstable and stable eigenprojections
associated with $A^0_\infty(\lambda)$.
\end{definition}

\section{Convergence to isolated point spectra}\label{s:convergence}

We begin by recovering in a particularly direct and simple fashion
the basic result of Gardner \cite{G2} on bifurcation from isolated point spectra;
for related arguments, see \cite{OZ1,SS,Z1}.

\subsection{Convergence as $X^\eps\to \infty$}\label{s:conv}
To show convergence, we first reformulate the homoclinic Evans
function as a Jost-function type determinant
such as appears in the definition of the periodic Evans function,
involving the difference of two matrix-valued solutions.
See \cite{GM,Z3} for related discussion.

\begin{lemma}\label{sublem}
Assuming (H1)--(H3), 
for $\bp L^-_\infty\\L^+_\infty\ep:=(R^-_\infty,R^+_\infty)^{-1}$,
\be\label{altevans}
D^0_r(\lambda)=(-1)^{n_r}\det( R^-L^-_\infty - R^+L^+_\infty)|_{x=0}.
\ee
\end{lemma}

\begin{proof}
Factoring
$( R^-L^-_\infty - R^+L^+_\infty)=
( R^-,-R^+)\bp L^-_\infty\\ L^+_\infty \ep
=
( R^-,-R^+)(R^-_\infty,  R^+_\infty )^{-1}
$,
taking determinants, and comparing to \eqref{e:hevans}, we obtain the result.
\end{proof}

\begin{proposition} \label{c:conv}
Assuming (H1)--(H3), on each compact $K\subset \Lambda_r$,
there exist $C,\theta>0$ such that
\be\label{e:Dconv}
|D^\eps_r(\lambda,\gamma) -D^0_r(\lambda)|\le Ce^{-\eta X^\eps/2}
\; 
	\hbox{\rm for all $\lambda \in K$, $|\gamma|=1$},
\ee
for any $\eta$ less than the minimum of $\bar \eta$, given in Lemma \ref{conjlem}, and the spectral gap 
of $A^0_\infty(\lambda)$, defined as the minimum absolute value of the real parts of the eigenvalues of $A^0_\infty$.
\end{proposition}

\br\label{sharprate}
In the generic case discussed in Remark \ref{h3rmk1}, the restrictions on $\eta$ in Proposition \ref{c:conv}
reduce to $0<\eta<$ spectral gap of $A^0_\infty(\lambda)$, by the estimates of \cite[Prop. 5.1]{SS}.
\er

\begin{proof}
Using the description of Lemma \ref{conjlem}, 
and the spectral expansion formula
$$
e^{A^0_\infty x}= 
e^{A^0_\infty \Pi_u x} R^-_\infty L^-_\infty 
+e^{A^0_\infty  \Pi_s x} R^+_\infty L^+_\infty ,
$$
we find, using $P^\eps_\pm(\pm X^\eps/2)=\Id$ 
(see \eqref{Pdecay}) 
and $ |e^{A^0_\infty  \Pi_s X^\eps/2} | \le Ce^{-\eta X^\eps/2}$, that
$$
\begin{aligned}
\Psi^\eps(0,\lambda)\Psi^\eps(-X^\eps/2,\lambda)^{-1}&=
P^\eps_-(0) \Big(P^\eps_-(-X^\eps/2)e^{-A^0_\infty X^\eps/2}\Big)^{-1}\\
&=
P^\eps_-(0) e^{A^0_\infty X^\eps/2} 
\\
&=
P^\eps_-(0) e^{A^0_\infty \Pi_u X^\eps/2} R^-_\infty L^-_\infty 
+O(e^{-\eta X^\eps/2})
\end{aligned}
$$
and, likewise, 
$
\Psi^\eps(0,\lambda)\Psi^\eps(X^\eps/2,\lambda)^{-1}=
P^\eps_+(0) 
e^{-A^0_\infty \Pi_s X^\eps/2} R^+_\infty L^+_\infty
+O(e^{-\eta X^\eps/2}),
$
from which we find, factoring similarly as in the proof of 
Lemma \ref{sublem}, that
\begin{equation}\label{exp}
\begin{aligned}
D^\eps_r(\lambda, \gamma)&= 
	e^{-\tr A^0_\infty \Pi_u X^\eps/2} e^{\tr A^0_\infty \Pi_s X^\eps/2} (-\gamma)^{-n_r}
\det \bp P^\eps_-(0)R^-_\infty & P^\eps_+(0)R^+_\infty \ep \\
	&\quad \times
	\det \Big(
	\bp L^-_\infty e^{A^0_\infty \Pi_u X^\eps/2}R^-_\infty & 0\\ 0 &-\gamma L^+_\infty e^{-A^0_\infty \Pi_s X^\eps/2}R^+_\infty \ep
	+O(e^{-\eta X^\eps/2})
	\Big)
	\\
	&\quad \times
	\det \bp L^-_\infty 
	\\ L^+_\infty 
	\ep 
	\\
	&=
\det \bp P^\eps_-(0)R^-_\infty & P^\eps_+(0)R^+_\infty \ep 
	\det ( \Id + O(e^{-\eta X^\eps}))
	\det \bp L^-_\infty 
	\\ L^+_\infty 
	\ep 
	\\
	&= D^0_r(\lambda)+O(e^{-\eta X^\eps/2}),
\end{aligned}
\end{equation}
by 
$\bp L^-_\infty \\L^+_\infty\ep:=\bp R^-_\infty & R^+_\infty\ep^{-1}$
and the definition of $D^0_r$ in \eqref{e:hevans}.
\end{proof}

\begin{cor}[\cite{G2,SS,OZ1}]\label{c:loops}
Assuming (H1)--(H3), on compact
$K\subset \Lambda_r$ such that $D^0_r$ does not vanish on $\partial K$,
the spectra of $L^\eps$ for $X^\eps$ sufficiently large
consists of loops of spectra $\lambda_{r,k}^\eps(\gamma)$,
$k=1,\dots, m_r$, within $O(e^{-\eta X^\eps/2m_r})$ of the 
eigenvalues $\lambda_r$ of $L^0$, where $m_r$ denotes the multiplicity
of $\lambda_r$
and $\eta$ is as in \eqref{e:Dconv}.
\end{cor}

\begin{proof}
Immediate by properties of analytic functions (Rouch\'es Theorem).
\end{proof}

\br \label{manyD}
\textup{
Note that different rescalings of $\tilde{E}^\eps$ converge as $X^\eps\to \infty$
to different versions $D^0_r$ of the homoclinic Evans function on different
components $\Lambda_r$.
}
\er

\subsection{A flip-type stability index and behavior near $\lambda=0$}\label{s:translational}
We mention in passing the important special case 
of an isolated eigenvalue at $\lambda=0$ of the limiting homoclinic wave,
corresponding with translational invariance of the underlying PDE.
Similar translational ($\gamma=1$)-eigenvalues occur at $\lambda=0$ for periodic waves of all periods $X^\eps$.
As shown in \cite{SS} by rather different Melnikov integral/Lyapunov-Schmidt computations,
this exact correspondence for $\gamma=1$ implies cancellation in $D^\eps_r -D^0_r$,
yielding convergence at faster exponential rate $O(e^{-\alpha \eta X^\eps/2})$, where $\alpha>1$,
and also an asymptotic description of the location of $(\gamma\neq 1$)-eigenvalues near $\lambda=0$, deciding 
diffusive spectral stability of spectral loops passing through the origin.
In typical cases, these loops are to lowest order in $|\lambda-\lambda_*|$
ellipses with axes parallel to real and imaginary coordinate axes,
hence their diffusive stability or instability is decided by whether the ($\gamma=-1)$-eigenvalue 
lies in the stable ($\Re \lambda <0$) or unstable ($\Re \lambda>0$) half-space \cite[Discussion, p. 182, par. 2]{SS}.

These conclusions do not follow by our straightforward computations above, and indeed would
appear to be cumbersome to reproduce by such an Evans function approach.
However, a related {\it necessary} condition $\sigma\geq 0$ based on the stability index
\be\label{stabindex}
\sigma:= \sgn E^\eps(0, -1) \sgn E^\eps(\infty_{real},-1),
\ee
is readily obtained from the Evans function formulation by the observation that
$E^\eps(\lambda,-1)$ is real-valued for $\lambda$ real, hence $\delta\leq 0$ by the intermediate value theorem
implies existence of a $-1$-eigenvalue with nonnegative real part, violating diffusive stability conditions
(D2)-(D3).

The necessary condition \eqref{stabindex} is valid in much more general contexts than is the necessary and 
sufficient condition obtained in \cite{SS}, in particulat to systems with conservation laws for which $\lambda=0$
is an embedded eigenvalue of the limiting homoclinic wave.
See, for example, \cite[Thm 1.9]{JNRYZ} for an important application of this principle
in the case of the Saint Venant equations of inclined shallow water flow.
A change of sign in $\delta$ corresponds to passage of a $-1$-eigenvalue through $\lambda=0$, or ``flip'' bifurcation
in the periodic traveling-wave ODE.

\section{Convergence to essential spectra}\label{s:essential}

We now turn to our main object, of bifurcation from
essential homoclinic spectra of periodic spectra in the large-period limit.
Recall that the essential spectrum of the homoclinic limit is given by the union
of curves $\lambda_j(k)$ determined by the dispersion relation $ik\in \sigma(A^0_\infty(\lambda))$,
$k\in \RR$, bounding domains of hyperbolicity $\Lambda_r$
with Evans functions $D^0_r$.

\subsection{Convergence to an isolated arc of essential spectra}\label{s:arc}
Consider the generic situation of
a point of homoclinic essential spectrum $\lambda_*=\lambda_j(k_*)$ lying on a single curve $\lambda_j$, for
which $k_*$ is a nondegenerate root and unique solution of $\lambda_j(k)=\lambda_*$.
Without loss of generality, suppose that $\lambda_j$ separates domains of hyperbolicity $\Lambda_1$ and
$\Lambda_2$, on which are defined homoclinic Evans functions $D^0_1$ and $D^0_2$, as described in \eqref{e:hevans}.

By assumption, $\mu_* :=ik_*$ is a simple imaginary eigenvalue of $A^0_\infty(\lambda_*)$, 
with all other eigenvalues of strictly positive or negative real part.
By standard matrix perturbation theory, therefore, there exist near $\lambda_*$
an analytic eigenvalue $\mu_c(\lambda)$ and associated eigenprojecton $\Pi_c(\lambda)$
with $\mu_c(\lambda_*)=ik_*$ and analytic strongly stable and unstable eigenprojections $\Pi_{ss}$ and $\Pi_{su}$,
corresponding at $\lambda=\lambda_*$ to center, stable, and unstable projections of $A^0_\infty(\lambda_*)$.
Without loss of generality, suppose that $\Re \mu_c<0$ on $\Lambda_2$ and
$\Re \mu_c>0$ on $\Lambda_1$.
By analyticity of $\Pi_c$ , $\Pi_{ss}$, and $\Pi_{su}$, 
the homoclinic Evans functions $D^0_1$ and $D^0_2$ extend analytically to a neighborhood of  $\lambda_*$; 
denote their values at $\lambda_*$ as $d_1=D^0_1(\lambda_*)$ and $d_2=D^0_2(\lambda_*)$.

\begin{definition}\label{e:blowup}
Near $\lambda_*$, we define the {\it transitional periodic Evans function} as
\be\label{e:transper}
\tilde D^\eps(\lambda, \gamma):=
e^{-\tr A^0_\infty \Pi_{su} X^\eps/2}
e^{\tr A^0_\infty \Pi_{ss} X^\eps/2}
(-\gamma)^{-n_1} \tilde E(\lambda, \gamma) ,
\ee
where $\Pi_{su}$ and $\Pi_{ss}$ denote strongly unstable and stable eigenprojections
of $A^0_\infty(\lambda)$ ($n_1$ here corresponding to $\dim \Range \Pi_{ss}$), and
the transitional homoclinic Evans function as
\be\label{e:transhom}
\tilde D^0(\lambda, \gamma):=
e^{\mu_c(\lambda) X^\eps/2} D^0_1(\lambda) -\gamma e^{-\mu_c(\lambda) X^\eps/2}D^0_2(\lambda) .
\ee
\end{definition}

\begin{proposition}\label{main}
	Assuming (H1)--(H3), let $\lambda_*=\lambda_j(k_*)$ lie on 
	an isolated arc of homoclinic essential spectra as described above.
Then, there exists $C_1>0$ such that for any $C>0$, for $X^\eps$ sufficiently large, 
\be\label{e:Dtildeconv}
|\tilde D^\eps(\lambda,\gamma) - \tilde D^0(\lambda,\gamma)|\le Ce^{-\eta X^\eps}/2, \, \eta>0, \;
	\hbox{\rm for all $\lambda \in B(\lambda_*, C/X^\eps)$,  $|\gamma|=1$}.
\ee
\end{proposition}

\begin{proof}
	Letting $ R^-_\infty $,$ R^c_\infty $, and $ R^+_\infty$ denote bases of 
	$\Range \Pi_{su}$, $\Range \Pi_{c}$, and $\Range \Pi_{ss}$, set
	$$
	\bp L^-_\infty \\ L^c_\infty\\L^+_\infty\ep= \bp R^-_\infty & R^c_\infty & R^+_\infty\ep^{-1}.
	$$
Expanding as in the proof of Proposition \ref{c:conv} via spectral resolution of $A^0_\infty$ gives
	$$
\begin{aligned}
	\tilde D^\eps &(\lambda, \gamma)= 
	e^{-\tr A^0_\infty \Pi_{su} X^\eps/2} e^{\tr A^0_\infty \Pi_{ss} X^\eps/2} (-\gamma)^{-n_1}\\
	&\quad \times
	\det \Big(
	P^\eps_-(0)R^-_\infty e^{A^0_\infty \Pi_{su} X^\eps/2} L^-_\infty 
	+
	P^\eps_-(0)R^c_\infty e^{\mu_c(\lambda ) X^\eps/2} L^c_\infty 
	\\
	&\quad - \gamma P^\eps_+(0)R^c_\infty e^{-\mu_c(\lambda)  X^\eps/2} L^c_\infty 
	- \gamma P^\eps_+(0)R^+_\infty e^{-A^0_\infty \Pi_{ss} X^\eps/2} L^+_\infty 
	+O(e^{-\eta X^\eps/2})\Big).
\end{aligned}
$$
Using $\Re \mu_c X^\eps=O(1)$ on $B(\lambda_*, C/X^\eps)$ and
	factoring as in the proof of Proposition \ref{c:conv} gives
$$
\begin{aligned}
\tilde D^\eps(\lambda, \gamma)&= 
	\det \bp P^\eps_-(0)R^-_\infty & 
	\big(e^{\mu_c(\lambda)X^\eps/2} P^\eps_-(0)R^c_\infty 
	- \gamma e^{-\mu_c(\lambda)X^\eps/2} P^\eps_+(0)R^c_\infty \big)
	&  P^\eps_+(0)R^+_\infty \ep \\
	&\quad \times
	\det \bp L^-_\infty  \\ L^c_\infty \\ L^+_\infty  \ep 
	+O(e^{-\eta X^\eps/2}),
\end{aligned}
$$
or, expanding the first determinant with respect to the middle column and recalling the definitions of $D^0_1$ and $D^0_2$,
$
\tilde D^\eps(\lambda, \gamma)= 
e^{\mu_c(\lambda) X^\eps/2} D^0_1(\lambda) 
- \gamma e^{-\mu_c(\lambda) X^\eps/2}D^0_2(\lambda) 
+ O(e^{-\eta X^\eps/2}).
$
\end{proof}

\begin{cor}\label{c:arc}
	Assuming (H1)--(H3), let $\lambda_*=\lambda_j(k_*)$ 
	lie on an isolated arc of homoclinic essential spectra as 
	above,
with $D^0_1(\lambda_*)$, $D^0_2(\lambda_*)\neq 0$.
	Then, denoting $\gamma=e^{ikX^\eps}$, there exists $C_1>0$ such that for any $C>0$, for $X^\eps$ sufficiently large, 
	the $\gamma$-eigenvalues of $\bar u^\eps$ in $B(\lambda_*, C/X^\eps)$,
	corresponding to zeros of $E^\eps(\cdot, \gamma)$, lie within $C_1/X^\eps$ of the set 
\be\label{latice}
\{\lambda_j(\kappa): \, \hbox{\rm $\kappa= k$ mod ($2\pi/ X^\eps$)} \}.
	\ee
\end{cor}

\begin{proof}
From \eqref{e:transhom}-\eqref{e:Dtildeconv} and Taylor expansion
$ \mu_c(\lambda_*+z/X^\eps)=ik_*+ \mu_c'(\lambda_*)z/X^\eps + O(|X^\eps|^{-2}), $
we see immediately that for $z\in B(0,C)$, $|\hat \gamma|=1$ fixed
	$$
	d_1^{-1} e^{-ik_* X^\eps/2 +\mu_c'(\lambda_*) z /2}
	\tilde D^\eps( \lambda_*+ z/X^\eps,e^{ik_*X^\eps}\hat \gamma)\to 
		e^{\mu_c'(\lambda_*) z }  - \hat \gamma d_2/d_1
	$$
as $\eps \to 0$, at rate 
$O(|X^\eps|^{-1})$, from which we find by properties of analytic functions (Rouch\'es Theorem) that zeros of 
$\tilde D^\eps( \lambda_*+ z/X^\eps,e^{ik_*X^\eps}\hat \gamma)$ converge at rate $O(|X^\eps|^{-1})$ to solutions of
\be\label{e:sharp}
\mu_c'(\lambda_*)z= [ \ln (\hat \gamma)+ \ln (d_2/d_1)] \; \hbox{\rm mod($2\pi i$)}.
\ee

Converting back to $\lambda$ coordinates, we have that $\lambda-\lambda_*=z/X^\eps$ converges at 
rate $O(|X^\eps|^{-2})$ to 
\be\label{lest}
(1/\mu_c'(\lambda_*))
(\ln(\hat \gamma) + \ln (d_2/d_1)) /X^\eps \; \hbox{\rm mod($2\pi i/X^\eps$)},
\ee
whence, using $\hat \gamma= \gamma e^{-ik_*X^\eps}= e^{ikX^\eps -ik_*X^\eps}$, or
$\ln(\hat \gamma)/X^\eps=ik-ik_*$ (mod $2\pi i/X^\eps$), we find that
$\mu_c(\lambda)= ik_* + \mu_c'(\lambda_*)(\lambda-\lambda_*) + O(|X^\eps|^{-2})$ converges at rate $O(|X^\eps|^{-2})$ to
$$
[ik + \ln (d_2/d_1)/X^\eps ]  \; \hbox{\rm mod($2\pi i/X^\eps$)},
$$
and thus at $O(|X^\eps|^{-1})$ rate to $ik$ mod($2\pi i/X^\eps$).
By definition of $\lambda_j$, this is equivalent to convergence of $\lambda$ at rate $O(|X^\eps|^{-1})$ to $\lambda_j(\kappa)$
for $ \kappa=k$ mod($2\pi/X^\eps$).
Noting that convergence at each step of the argument is uniform with respect to $\hat \gamma$, we obtain the result.
\end{proof}

\br\label{sharprmk}
Estimate \eqref{lest} shows that the rate of convergence $O(|X^\eps|^{-1})$ is sharp unless $d_1=d_2$, giving
an explicit corrector $\ln(d_2/d_1)/\mu_c'(\lambda_*) X^\eps$ valid to order $O(|X^\eps|^{-2})$.
\er

\subsection{Convergence to embedded point spectra}\label{s:embedded}

Finally, we consider the case of an embedded homoclinic eigenvalue $\lambda_*$ of multiplicity $m$
contained in an isolated arc of essential spectrum $\lambda_j$ 
dividing regions of hyperbolicity $\Lambda_1$ and $\Lambda_2$ as described in Section \ref{s:arc}.
By multiplicity $m$ eigenvalue, we mean a value with an $m$-dimensional subspace of decaying generalized eigenfunctions.
This implies that both homoclinic Evans functions $D^0_1$ and $D^0_2$ have a zero of multiplicity 
at least $m$ at $\lambda=\lambda_*$, since they differ only with respect to the nondecaying mode $\mu_c$.
We make the additional nondegeneracy assumption that 
$D^0_1$ and $D^0_2$ have zeros at $\lambda_*$ of exactly multiplicity $m$.

\begin{cor}\label{c:arcembed}
Assuming (H1)--(H3), let $\lambda_*=\lambda_j(k_*)$ 
be a point lying on an isolated arc of homoclinic essential spectra 
at which $D^0_1$, $D^0_2$ possess zeros of degree $m$. 
Then, denoting $\gamma=e^{ikX}$, there exists $C_1>0$ such that for any $C>0$, $\eta>\tilde \eta>0$, 
for $X^\eps$ sufficiently large, 
the $\gamma$-eigenvalues of $\bar u^\eps$ in $B(\lambda_*, C/X^\eps)$,
	or zeros of $E^\eps(\cdot, \gamma)$, 
	consist of points lying within $ C_1/X^\eps$ of
$\{\lambda_j(\kappa): \, \hbox{\rm $\kappa= k$ mod ($2\pi/ X^\eps$)} \}$
plus $m$ points lying within $C_1e^{-\tilde \eta X^\eps/2(m+1)}$ of $\lambda_*$.
\end{cor}

\begin{proof}
Factoring $D^0_1(\lambda)=(\lambda-\lambda_*)^m \hat D^0_1(\lambda)$, 
$D^0_2(\lambda)=(\lambda-\lambda_*)^m \hat D^0_2(\lambda)$, setting
$\hat d_1:=\hat D^0_1(\lambda_*)$, $\hat d_2:=\hat D^0_2(\lambda_*)$, and
and applying again \eqref{e:Dtildeconv}, we find that for $z\in B(0,C)$, $|\hat \gamma|=1$ fixed,
	$$
	(X^\eps)^m \hat d_1^{-1} e^{-ik_* X^\eps/2 +\mu_c'(\lambda_*) z /2}
	\tilde D^\eps( \lambda_*+ z/X^\eps,e^{ik_*X^\eps}\hat \gamma)\to 
z^m \big( e^{ \mu_c'(\lambda_*) z }  - \hat \gamma \hat d_2/\hat d_1\big)
	$$
as $\eps \to 0$, at rate $O(|X^\eps|^{-1})$, from which we find that the zeros of 
$\tilde D^\eps( \lambda_*+ z/X^\eps,e^{ik_* X^\eps}\hat \gamma)$ converge at rate 
$O(|X^\eps|^{-1/(m+1)})$ to solutions of 
	\be\label{rescale}
	\mu_c'(\lambda_*)z= [ \ln(\hat \gamma) + \ln (\hat d_2/\hat d_1)] \, \hbox{\rm mod($2\pi i$)},
	\ee
and to the $m$-tuple root at $z=0$.  Converted back to $\lambda$-coordinates, this 
gives the desired result of $O(|X^\eps|^{-1})$ convergence to 
$\mathcal{L}_k=\{\lambda_j(\kappa): \, \hbox{\rm $\kappa= k$ mod ($2\pi/ X^\eps$)} \}$ along with the
suboptimal result of $O(|X^\eps|^{-1 + 1/(m+1)})$ convergence to the $m$-tuple root at $\lambda=\lambda_*$.

To obtain the optimal $O(e^{-\tilde \eta X^\eps/2(m+1)})$ rate of convergence stated 
for the $m$-fold eigenvalue $\lambda_*$, we may go back again to \eqref{e:Dtildeconv} to obtain the sharper result that
$$
	(X^\eps)^m (\hat D^0_1)^{-1} e^{-ik_* X^\eps/2 +(\mu_c(\lambda_*+ z/X^\eps)-\mu_c (\lambda_*))X^\eps  /2}
	\tilde D^\eps( \lambda_*+ z/X^\eps, e^{ik_* X^\eps}\hat \gamma) 
$$
lies within $O(e^{-\tilde \eta X^\eps/2})$ of
$ z^m \big( 
e^{(\mu_c(\lambda_* + z/X^\eps)- \mu_c(\lambda_*)) X^\eps }  
-  \hat \gamma \hat D^0_2/\hat D^0_1 \big)$,
from which we may obtain the result by direct application of Rouch\'es Theorem, on a case-by-case basis depending
whether or not $\hat d_2/\hat d_1=e^{ikX^\eps}$, i.e., whether or not $\lambda_* \in \mathcal{L}_k$.
We omit the details.
\end{proof}

\subsection{Behavior near an embedded eigenvalue at $\lambda=0$}\label{s:embedzero}
In the case of a multiplicity-one embedded ``translational'' homoclinic eigenvalue at $\lambda=0$,
it often transpires that, besides the corresponding translational ($\gamma=1$)-eigenvalues of nearby periodic waves
at $\lambda=0$, the ($\gamma=1$)-eigenvalue at $\lambda=0$ has additional multiplicity equal to the number of arcs
$\lambda_j$ of homoclinic essential spectra on which it lies.
See in particular the case of hyperbolic and parabolic balance laws discussed in \cite{JZN,JNRZ1}.
In this case we may deduce from \eqref{rescale} that $\hat d_1=\hat d_2$, giving an improved convergence
rate of $(C_1/X^\eps)^2$ of periodic to homoclinic essential spectra as described in Corollary \ref{c:arcembed}.

\section{Application to Saint Venant equations}\label{s:SV}
The Saint Venant equations for inclined shallow-water flow are, in nondimensional form,
\begin{equation}\label{sv}
\displaystyle
\partial_t h+\partial_x (hu)=0,\quad \partial_t (hu)+\partial_x \left(hu^2+\frac{h^2}{2F^2}\right)=h-|u| u
+\nu\partial_x(h\partial_x u),
\end{equation} 
where $h$ is fluid height, $u$ vertical fluid velocity average,
$x$ longitudinal distance, $t$ time,
$F$ a Froude number given by the ratio between (a chosen reference) speed of the fluid 
and speed of gravity waves, and $\nu=R_e^{-1}$, with $R_e$ the Reynolds number of the fluid.
Terms $h$ and $|u|\,u$ on the right represent opposing forces of gravity and turbulent bottom friction.

Well-known solutions of \eqref{sv} are periodic {\it roll} (traveling) {\it waves} \cite{Br1,Br2,Dr} advancing with constant speed
down a canal or spillway. For fixed Froude number $F>2$, these appear in families indexed by period $X$ and average height $u$ over one period, arising through a classical Hopf to homoclinic bifurcation scenario \cite{BJNRZ4}.
In particular, they feature a homoclinic limit as studied in this note, with a single embedded eigenvalue at $\lambda=0$, contained in a single arc of unstable essential spectra, and all other spectra strictly stable; see \cite{BJRZ} for further details.
As noted in the introduction, this is a case to which the results of \cite{G1,G2,SS} do not apply but that can be
treated by our analysis here.

Specifically,
Corollary \ref{c:arcembed} verifies the intuitive conclusion that periodic roll waves are unstable in the large-period limit
due to convergence of periodic spectra
to unstable homoclinic essential spectra, settling the question of large-period stability.
However, there is a much more interesting phenomenon involved with the homoclinic limit, worthy of further investigation.
Namely, as noted in \cite{PSU} more generally, mathematical models of inclined thin film flow appear to share
the feature that homoclinic waves are essentially unstable; yet,
both experiment and models of inclined thin film flow yield asymptotic behavior consisting of the approximate superposition
of well-separated homoclinic waves.

A heuristic explanation of this paradox given in \cite{BJNRZ4,BJRZ} is that sufficiently closely arrayed homoclinic waves can stabilize each others convective essential instabilities, manifested as exponentially growing perturbations traveling with a 
nonzero group velocity, through de-amplifying properties of the localized homoclinic pulses encoded in their strictly stable point spectrum.
Indeed, this appears to match well with observed onset of stability of periodic waves at periods suggested by this proposed mechanism \cite[Section 6]{BJRZ}.

However, up to now, there is lacking
a rigorous explanation at the level of spectral stability
relating the observed behavior to properties of the homoclinic wave.  To carry out such an
analysis by asymptotic techniques like those used here and in \cite{G1,G2,SS}
seems an outstanding open problem in periodic stability theory and dynamics of thin-film flow.


\appendix

\section{Asymptotic ODE theory}\label{asymptotic}
Here, we recall the asymptotic Evans function results cited earlier;
for proofs, see, e.g., \cite{Z2}.

\subsection{The conjugation lemma}\label{s:conj}
Consider a general first-order system
\be \label{gen1}
W'=A^p(x,\lambda)W
\ee
with asymptotic limits $A^p_\pm$ as $x\to \pm \infty$,
where $p\in \RR^m$ denote model parameters.

\bl [\cite{MZ,PZ}]\label{conjlemapp}
Suppose for fixed $\theta>0$ and $C>0$ that 
\be\label{udecay}
|A^p-A^p_\pm|(x,\lambda)\le Ce^{-\theta |x|}
\ee
for $x\gtrless 0$ uniformly for $(\lambda,p)$ in a neighborhood of 
$(\lambda_0)$, $p_0$ and that $A$ varies analytically in $\lambda$ 
and continuously in $p$ as a function into $L^\infty(x)$.
Then, there exist in a neighborhood of $(\lambda_0,p_0)$
invertible linear transformations $P^p_+(x,\lambda)=I+\Theta_+^p(x,\lambda)$ 
and $P_-^p(x,\lambda) =I+\Theta_-^p(x,\lambda)$ defined
on $x\ge 0$ and $x\le 0$, respectively,
analytic in $\lambda$ and continuous in $p$ 
as functions into $L^\infty [0,\pm\infty)$, such that
\begin{equation}
\label{Pdecayapp} 
| \Theta_\pm^p |\le C_1 e^{-\bar \theta |x|}
\quad
\text{\rm for } x\gtrless 0,
\end{equation}
for any $0<\btheta<\theta$, some $C_1=C_1(\bar \theta, \theta)>0$,
and the change of coordinates $W=:P^p_\pm Z$ reduces \eqref{gen1} to 
the constant-coefficient limiting systems
\begin{equation}
\label{glimitapp}
Z'=A^p_\pm Z 
\quad
\text{\rm for } x\gtrless 0.
\end{equation}
\el

\br\label{pfrmk}
As shown in the proof (e.g., \cite{Z2}), necessarily also $(P^p_\pm) '= A^p P_\pm  - P_\pm  A^p_\pm$.
\er

\subsection{The convergence lemma}\label{s:convlem}
Consider a family of first-order equations 
\be \label{gen2}
W'=A^p(x,\lambda)W
\ee
indexed by a parameter $p$, and satisfying exponential
convergence condition \eqref{udecay} uniformly in $p$.
Suppose further that, for some $\delta(p)\to 0$ as $p\to 0$,
\begin{equation} \label{residualest}
|(A^p- A^p_\pm)-
(A^0- A^0_\pm)|
\le C\delta(p) e^{-\theta |x|}, \qquad \theta>0
\end{equation}
and
\begin{equation} \label{newest}
|( A^p- A^0)_\pm)| \le C\delta(p).
\end{equation}

\begin{lemma}[\cite{PZ,BHZ}] \label{evanslimit}
Assuming \eqref{udecay} and \eqref{residualest}--\eqref{newest},
for $|p|$ sufficiently small,
there exist invertible linear transformations 
$P_+^p(x,\lambda)=I+\Theta_+^p(x,\lambda)$ 
and $P_-^0(x,\lambda) =I+\Theta_-^p(x,\lambda)$ defined
on $x\ge 0$ and $x\le 0$, respectively,
analytic in $\lambda$ as functions into $L^\infty [0,\pm\infty)$, such that
\begin{equation}
\label{cPdecayapp} 
| (P^p-P^0)_\pm(x) |\le C_1 \delta(p) e^{-\bar \theta |x|}
\quad
\text{\rm for } x\gtrless 0,
\end{equation}
for any $0<\btheta<\theta$, some $C_1=C_1(\bar \theta, \theta)>0$,
and the change of coordinates $W=:P_\pm^p Z$ reduces \eqref{gen2} to 
the constant-coefficient limiting systems
\begin{equation}
\label{cglimit}
Z'=A^p_\pm(\lambda) Z 
\quad
\text{\rm for } x\gtrless 0.
\end{equation}
\end{lemma}

\end{document}